\documentclass[twoside, a4paper,12pt]{amsart}

\usepackage{tikz}
\usepackage{hyperref} 
\usepackage{amsfonts}
\usepackage{amsthm}
\usepackage{amsmath}
\usepackage[hmarginratio=1:1]{geometry}

\theoremstyle{plain}
\newtheorem*{claim*}{Claim}
\newtheorem{thm}{Theorem}[section]
\newtheorem{corollary}[thm]{Corollary}
\newtheorem{lemma}[thm]{Lemma}
\newtheorem{prop}[thm]{Proposition}
\theoremstyle{definition}
\newtheorem{defn}[thm]{Definition}
\newtheorem{ex}[thm]{Example}

\newtheorem{remark}[thm]{Remark}
\newtheorem{con}[thm]{Construction}

\begin{document}
\title{\large{An introduction to presentations of monoid acts: quotients and subacts}}
\author{Craig Miller and Nik Ru{\v s}kuc}
\address{School of Mathematics and Statistics, St Andrews, Scotland, UK, KY16 9SS}
\email{cm380@st-andrews.ac.uk, nik.ruskuc@st-andrews.ac.uk}

\begin{abstract}
The purpose of this paper is to introduce the theory of presentations of monoids acts.
We aim to construct `nice' general presentations for various act constructions pertaining to subacts and Rees quotients.
More precisely, given an $M$-act $A$ and a subact $B$ of $A$, on the one hand we construct presentations for $B$ and the Rees quotient $A/B$ using a presentation for $A$,
and on the other hand we derive a presentation for $A$ from presentations for $B$ and $A/B$.
We also construct a general presentation for the union of two subacts.
From our general presentations, we deduce a number of finite presentability results.
Finally, we consider the case where a subact $B$ has finite complement in an $M$-act $A$.
We show that if $M$ is a finitely generated monoid and $B$ is finitely presented, then $A$ is finitely presented.
We also show that if $M$ belongs to a wide class of monoids, including all finitely presented monoids, then the converse also holds.
\end{abstract}

\maketitle

\section{\large{Introduction}\nopunct}

The concept of presentations is significant within many areas of algebra.
Finite presentability of acts was first studied by P. Normak in 1977 \cite{Normak}, and is a fundamental finiteness condition for the theory of monoid acts (see \cite{Kilp}).
The related notion of {\em coherency} for monoids was introduced by V. Gould in 1992 \cite{Gould1},
and has since been intensively studied by several authors (see \cite{Gould2}, \cite{Gould3}).
Finite presentability of acts also plays a key role in the monoid properties of being {\em right Noetherian} \cite{Normak} and being {\em completely right pure} \cite{Gould}.
However, there has not yet been developed a systematic theory of presentations for acts over monoids.
This paper is concerned with introducing such a theory through considering presentations for two of the most basic constructions: quotients and subacts.
A follow-on article will deal with various product constructions of acts.\par
The paper is structured as follows.
In Section 2, we collect some basic definitions and facts about acts.
In Section 3, we introduce the notions of presentations and finite presentability for a monoid act, provide various examples of act presentations, 
and record several results which will be of vital importance in the rest of the paper.
In the remainder of the paper, we study presentations for various constructions.
Typically we first obtain a general presentation for a construction and then derive corollaries regarding finite presentability.
Section 4 is concerned with presentations of Rees quotients.
Before moving to presentations of subacts in general in Section 6, we first discuss presentations of unions of subacts in Section 5.
Section 6 splits into two parts;
in the first part we construct a general presentation for a subact, and in the second part we study a specific case where the subact has finite complement.

\section{\large{Preliminaries}\nopunct}

The theory of monoid acts is essentially the theory of representations of monoids by transformations.
A monoid is a semigroup with an identity.
One of the most common and universal ways of defining monoids is by means of presentations, and we briefly review the basics here.
We refer the reader to \cite[Section 1.6]{Howie} for a more systematic introduction, and to \cite{Ruskuc1} for a more detailed development.\par
Let $Z$ be an alphabet.  We denote by $Z^{\ast}$ the monoid of all words in $Z$.
A {\em monoid presentation} is a pair $\langle Z\,|\,P\rangle_{\text{Mon}}$, where $P\subseteq Z^{\ast}\times Z^{\ast}$.\par
A monoid $M$ is said to be {\em defined} by the monoid presentation $\langle Z\,|\,P\rangle_{\text{Mon}}$ if $M\cong Z^{\ast}/\rho$, where $\rho$ is the smallest congruence on $Z^{\ast}$ containing $P$.
Thus we can identify $M$ with $Z^{\ast}/\rho$, so that the elements of $M$ are the $\rho$-classes of words from $Z^{\ast}$.
To put it differently, each word $w\in Z^{\ast}$ {\em represents} an element of $M$.\par
Let $u, v\in Z^{\ast}$.  We say that $v$ is obtained from $u$ by an {application of a relation from} $P$ if $u=pqr$ and $v=pq^{\prime}r$,
where $p, r\in Z^{\ast}$ and $(q, q^{\prime})\in P$.
We say that $u=v$ is a {\em consequence} of $P$ if either $u$ and $v$ are identical words or if there exists a sequence
$u=w_1, w_2, \dots, w_k=v$ where each $w_{i+1}$ is obtained from $w_i$ by an application of a relation from $R$.
We have the following basic fact:
\begin{lemma}
Let $M$ be a monoid defined by a presentation $\langle Z\,|\,P\rangle_{\emph{Mon}}$, and let $u, v\in Z^{\ast}$.
Then $u=v$ holds in $M$ if and only $u=v$ is a consequence of $P$.
\end{lemma}
We now proceed to present the basic concepts of the theory of monoid acts.
For more details, see \cite[Section 1.4]{Kilp}.\par 
Let $M$ be a monoid with identity 1. A {\em (right) $M$-act} is a non-empty set $A$ together with a map 
$$A\times M\to A, (a, m)\mapsto am$$
such that $a(mn)=(am)n$ and $a1=a$ for all $a\in A$ and $m, n\in M$.
For instance, $M$ itself is an $M$-act via right multiplication.\par
A subset $B$ of an $M$-act $A$ is a {\em subact} of $A$ if $bm\in B$ for all $b \in B$, $m\in M$.
Note that the right ideals of $M$ are precisely the subacts of the $M$-act $M$.\par
A subset $U$ of an $M$-act $A$ is a {\em generating set} for $A$ if for any $a\in A$, there exist $u\in U, m \in M$ such that $a=um$.
We write $A=\langle U\rangle$ if $U$ is a generating set for $A$.
An $M$-act $A$ is said to be {\em finitely generated} (resp. {\em cyclic}) if it has a finite (resp. one-element) generating set.\par
Note that a right ideal of $M$ can be generated by a set {\em as an $M$-act} or {\em as a semigoup}.
We introduce the convention that `generate' will always mean `generate as an $M$-act'.\par
For $M$-acts $A$ and $B$, a map $\theta : A \to B$ is an {\em $M$-homomorphism} if $(am)\theta=(a\theta)m$ for all $a\in A, m \in M$. 
If $\theta$ is also bijective, then it is an {\em $M$-isomorphism}, and we write $A\cong B$.\par
An equivalence relation $\rho$ on $A$ is an {\em ($M$-act) congruence} on $A$ if $(a, b)\in\rho$ implies $(am, bm)\in\rho$ for all $a, b\in A$ and $m\in M$.
For a congruence $\rho$ on an $M$-act $A$, the quotient set $A/\rho=\{[a] : a\in A\}$ becomes an $M$-act by defining $[a]m=[am]$.\par
Given an $M$-act $A$ and a subact $B$ of $A$, we define the {\em Rees congruence} $\rho_B$ on $A$ by 
$$a\rho_B b \iff a=b \text{ or } a, b\in B$$
for all $a, b\in A$.  We denote the quotient act $A/{\rho_B}$ by $A/B$ and call it the {\em Rees quotient} of $A$ by $B$.\par
For an $M$-act $A$ and $X\subseteq A\times A$, we denote by $\langle X\rangle_{\text{cg}}$ the smallest congruence on $A$ containing $X$.
A congruence $\rho$ on an $M$-act $A$ is {\em finitely generated} if there exists a finite subset $X\subseteq A\times A$ such that $\rho=\langle X\rangle_{\text{cg}}$.\par
Let $A$ be an $M$-act and let $X\subseteq A\times A$.  We introduce the notation
$$\overline{X}=X\cup\{(u, v) : (v, u)\in X\},$$
which will be used throughout the paper.
For $a, b\in A$, an $X${\em -sequence connecting} $a$ and $b$ is any sequence
$$a=p_1m_1, \; q_1m_1=p_2m_2, \; q_2m_2=p_3m_3, \; \dots, \; q_km_k=b,$$
where $(p_i, q_i)\in \overline{X}$ and $m_i\in M$ for $1\leq i\leq k$.\par
We now provide the following useful lemma (see \cite[Section 1.4]{Kilp} for a proof):
\begin{lemma}
Let $M$ be a monoid, let $A$ be an $M$-act, let $X\subseteq A\times A$ and let $a, b\in A$.
Then $(a, b)\in\langle X\rangle_{\text{cg}}$ if and only if either $a=b$ or there exists an $X$-sequence connecting $a$ and $b$.
\end{lemma}

A generating set $U$ for an $M$-act $A$ is a {\em basis} of $A$ if for any $a\in A$, there exist unique $u\in U$ and $m\in M$ such that $a=um$. 
An $M$-act $A$ is said to be {\em free} if it has a basis.
We have the following structure theorem for free acts.\par

\begin{prop}[{\cite[Theorem 1.5.13]{Kilp}}]
An $M$-act $A$ is free if and only if it is isomorphic to a disjoint union of $M$-acts all of which are $M$-isomorphic to $M$.
\end{prop}

This leads to the following explicit construction of a free act.

\begin{con}[{\cite[Construction 1.5.4]{Kilp}}]
Let $M$ be a monoid, let $X$ be a non-empty set, and consider the set $X\times M$.  
With the operation $$(x, m)n = (x, mn)$$ for all $(x, m)\in X\times M$ and $n\in M$, the set $X\times M$ is a free $M$-act with basis $X\times\{1\}$.
We denote this $M$-act by $F_{X, M}$, although we will usually just write $F_X$.
We will also usually write $x\cdot m$ for $(x, m)$ and $x$ for $(x, 1)$.
\end{con}

\begin{prop}[{\cite[Theorem 1.5.15]{Kilp}}]
Let $A$ be an $M$-act and let $F$ be a free $M$-act with basis $X$. 
If $\phi$ is any map from $X$ to $A$, then there exists a unique $M$-homomorphism $\theta : F\to A$ such that $\theta|_X=\phi$.
Further, if $X\phi$ is a generating set for $A$, then $\theta$ is surjective.
\end{prop}

\begin{corollary}
For any $M$-act $A$, there exists a free $M$-act $F$ such that $A$ is a homomorphic image of $F$.
\end{corollary}

\section{\large{Presentations of monoid acts}\nopunct}

\noindent We now introduce presentations of monoid acts.  The reader may also consult \cite[Section 1.5]{Kilp}.\par
Let $M$ be a monoid.  An {\em ($M$-act) presentation} is a pair $\langle X\,|\,R\rangle$, where $X$ is a non-empty set and $R \subseteq F_X \times F_X$ is a relation on the free $M$-act $F_X$.
An element $x$ of $X$ is called a {\em generator}, while an element $(u, v)$ of $R$ is called a {\em (defining) relation}, and is usually written as $u=v$.\par
An $M$-act $A$ is said to be {\em defined by the presentation} $\langle X\,|\,R\rangle$ if $A$ is $M$-isomorphic to the quotient act $F_X/\rho$, where $\rho=\langle R\rangle_{\text{cg}}$ is the smallest congruence on $F_X$ containing $R$.\par
Let $A$ be an $M$-act and $\theta : A\to F_X/\rho$ an $M$-isomorphism, where $\rho=\langle R\rangle_{\text{cg}}$.
We say an element $w\in F_X$ {\em represents} an element $a\in A$ if $a\theta=[w]_{\rho}$.\par
In the context of presentations, we write $w_1\equiv w_2$ if $w_1$ and $w_2$ are equal in $F_X$, and $w_1=w_2$ if they represent the same element of $A$.

\begin{remark}
Let $A$ be an $M$-act and let $X$ be any generating set for $A$.  
By Proposition 2.5, there exists a surjective $M$-homomorphism $\theta : F_X \to A$, so we have that $A\cong F_X/\text{ker~}\theta$ by the First Isomorphism Theorem for $M$-acts.  
Therefore, $A$ is defined by the presentation $\langle X\,|\,R\rangle$ where $R$ is any relation which generates $\text{ker~}\theta$.  
Hence, every $M$-act can be defined by a presentation.
\end{remark}

\begin{defn}
Let $\langle X\,|\,R\rangle$ be a presentation and let $w_1, w_2 \in F_X$. 
We say that the relation $w_1=w_2$ is a {\em consequence} of $R$ if $w_1\equiv w_2$ or there is an $R$-sequence connecting $w_1$ and $w_2$.\par
We say that $w_2$ is obtained from $w_1$ by an {\em application of a relation from} $R$ if there exists an $R$-sequence with only two distinct terms connecting $w_1$ and $w_2$.
\end{defn}

The next lemma follows immediately from Lemma 2.2.

\begin{lemma}
Let $\langle X\,|\,R\rangle$ be a presentation, let $A$ be the $M$-act defined by $\langle X \mid R\rangle$, and let $w_1, w_2 \in F_X$.  
Then $w_1=w_2$ in $A$ if and only if $w_1=w_2$ is a consequence of $R$.
\end{lemma}

Let $M$ be a monoid, let $A$ be an $M$-act generated by a set $Y$, and let $\phi : X\to Y$ be a surjective map.
Let $\theta : F_X\to A$ be the unique $M$-homomorphism extending $\phi$, and let $R$ be a subset of $F_X\times F_X$.
We say that $A$ {\em satisfies} $R$ (with respect to $\phi$) if for each $(u, v)\in R$, we have $u\theta=v\theta$; that is, $R\subseteq\text{ker~}\theta$.
Note that the $M$-act defined by a presentation $\langle X\,|\,R\rangle$ satisfies $R$.\par 
From the definition of an act defined by a presentation and Lemma 2.2, we have:

\begin{prop}
Let $M$ be a monoid, let $A$ be an $M$-act generated by a set $X$, and let $R\subseteq F_X\times F_X$.  Then $\langle X\,|\,R\rangle$ is a presentation for $A$ if and only if the following conditions hold:
\begin{enumerate}
 \item $A$ satisfies $R$;
 \item if $w_1, w_2 \in F_X$ such that $A$ satisfies $w_1=w_2$, then $w_1=w_2$ is a consequence of $R$.
\end{enumerate}
\end{prop}

The next fact follows from Proposition 2.5 and the Third Isomorphism Theorem for acts.

\begin{prop}
Let $A$ be an $M$-act defined by a presentation $\langle X\,|\,R\rangle$, let $B$ be an $M$-act, and let $\phi : X \to B$ be a map onto a generating set for $B$.
If $B$ satisfies $R$ (with respect to $\phi$), then there exists a surjective $M$-homomorphism $\psi : A \to B$. 
\end{prop}

\begin{defn}
A {\em finite presentation} is a presentation $\langle X\,|\,R\rangle$ where $X$ and $R$ are finite.  An $M$-act $A$ is {\em finitely presented} if it can be defined by a finite presentation.
\end{defn}

Note that a right ideal of a monoid $M$ may be finitely presented {\em as an $M$-act} or {\em as a semigoup}.
When we say that a right ideal is `finitely presented', we will always mean as an $M$-act.

\vspace{0.5em}
\begin{ex}~\par
\begin{enumerate}
 \item The free $M$-act $F_X$ is defined by the finite presentation $\langle X\,|\,~\rangle$.  
In particular, if $X$ is finite, then $F_X$ is finitely presented.
 \item For any monoid $M$, the $M$-act $M$ is finitely presented, since $M$ is a free $M$-act with basis $\{1\}$.
\end{enumerate}
\end{ex}

The following results are specialisations of well-known facts from general algebra.
They essentially reflect the fact that congruence-generation is an algebraic closure operator.
See, for instance, Section 1.5 and Theorem 2.5.5 in \cite{Burris} for more details.

\begin{prop}
Let $M$ be a monoid, let $A$ be an $M$-act defined by a finite presentation $\langle X\,|\,R\rangle$, and let $Y$ be another finite generating set for $A$.
Then $A$ can be defined by a finite presentation in terms of the generators $Y$.
\end{prop}

\begin{prop}
Let $M$ be a monoid, and let $A$ be a finitely presented $M$-act with a presentation $\langle X\,|\,S\rangle$ where $X$ is finite and $S$ is infinite.  
Then there exists a finite subset $S^{\prime}\subseteq S$ such that $A$ is defined by the finite presentation $\langle X\,|\,S^{\prime}\rangle$.
\end{prop}

\begin{corollary}\cite[Theorem 2]{Normak}.
Let $M$ be a monoid and let $A$ be a cyclic $M$-act.  
Then $A$ is finitely presented if and only if $A$ is isomorphic to a quotient act of $M$ by a finitely generated right congruence on $M$.
\end{corollary}

Let $M$ be a monoid with a generating set $S,$ and let $A$ be an $M$-act with a generating set $X$.
It can be easily proved, using Proposition 3.4, that the following are all presentations for $A$:
\begin{equation}
\langle A\,|\,a\cdot m=am~(a\in A, m\in M)\rangle;
\end{equation}
\begin{equation}
\langle X\,|\,x\cdot m=y\cdot n~(x, y\in X, m, n \in M, xm=yn)\rangle;
\end{equation}
\begin{equation}
\langle A\,|\,a\cdot s=as~(a\in A, s\in S)\rangle.
\end{equation}
The above relations should be interpreted in the relevant free acts.
Thus, for instance, the relation $a\cdot m=am$ in (1) stands for $(a, 1)m=(am, 1).$\par
Given the presentation (3), we immediately have the following:

\begin{lemma}
If $M$ is a finitely generated monoid and $A$ is a finite $M$-act, then $A$ is finitely presented.
\end{lemma}

If $M$ is a non-finitely generated monoid, however, then finite $M$-acts are not necessarily finitely presented, as the following example demonstrates.

\begin{ex}
Let $M=X^{\ast}$ be a free monoid with $X$ infinite, and consider the trivial $M$-act $A=\{0\}$.
Now $A$ is defined by the presentation 
$$\langle 0\,|\,0\cdot x=0~(x\in X)\rangle.$$
If $A$ were finitely presented, then it could be defined by a finite presentation 
$$P=\langle 0\mid 0\cdot x=0~(x\in X_0)\rangle,$$
where $X_0$ is a finite subset of $X$.
But for $x\not\in X_0$, the relation $0\cdot x=0$ is clearly not a consequence of the relations of $P$, so $A$ is not finitely presented.
\end{ex}

\begin{remark}
One may be tempted to think that the trivial $M$-act being finitely presented is equivalent to $M$ being finitely generated.
However, the trivial $M$-act is in fact finitely presented for a much larger class of monoids $M$.
For example, if $M$ is a monoid with a left zero $z$, it can be easily proved that the trivial $M$-act $\{0\}$ is defined by the finite presentation $\langle 0\,|\,0=0\cdot z\rangle$.
\end{remark}

The following lemma provides a necessary and sufficient condition for the trivial act to be finitely presented.

\begin{lemma}
Let $M$ be a monoid.  Then the trivial $M$-act $\{0\}$ is finitely presented if and only if there exists a finitely presented $M$-act $A$ which contains a zero.
\end{lemma}

\begin{proof}
The direct implication is obvious.
For the converse, let $A$ be an $M$-act with a zero $0$, and suppose that $A$ is defined by a finite presentation $\langle X\,|\,R\rangle$ where $0\in X$.
We define a finite set
$$R^{\prime}=\{0\cdot m=0\cdot n : (x\cdot m, y\cdot n)\in R\text{ for some } x, y\in X\},$$
and claim that $\{0\}$ is defined by the presentation $\langle 0\,|\,R^{\prime}\rangle$.
We need to show that for any $m\in M$, the relation $0\cdot m=0$ is a consequence of $R^{\prime}$.
Let $m\in M$.  Since $0\cdot m=0$ holds in $A$, it is a consequence of $R$, so we have an $R$-sequence connecting $0\cdot m$ and $0$.
Now, replacing every $x\in X$ appearing in this $R$-sequence with $0$, 
we obtain an $R^{\prime}$-sequence connecting $0\cdot m$ and $0$, so $0\cdot m=0$ is a consequence of $R^{\prime}$.
\end{proof}

Tietze transformations (for acts) provide a method for yielding a new presentation for a monoid act from a known presentation.
Given a presentation $\langle X\,|\,R\rangle$ for an $M$-act $A$, the {\em elementary Tietze transformations} are:
\begin{itemize}
\item[(T1)] adding new relations $u_i=v_i, i\in I,$ to $\langle X\,|\,R\rangle$, providing that each $u_i=v_i$ is a consequence of $R$;
\item[(T2)] deleting relations $u_i=v_i, i\in I,$ from $R$, providing that each $u_i=v_i$ is a consequence of $R\setminus\{u_i=v_i : i\in I\}$;
\item[(T3)] adding new generating elements $y_i, i\in I,$ and new relations $y_i=w_i, i\in I,$ to $\langle X\,|\,R\rangle$, for any $w_i\in F_X$;
\item[(T4)] if $\langle X\,|\,R\rangle$ has relations $x_i=w_i, i\in I,$ where $x_i\in X$ and $w_i\in F_{X^{\prime}}$ where $X^{\prime}=X\setminus\{x_i : i\in I\}$, then deleting each $x_i$ from $X$, deleting each $x_i=w_i$ from $R$, and replacing all remaining appearances of $x_i$ with $w_i$.
\end{itemize}

\begin{prop}
Two presentations define the same $M$-act if and only if one can be obtained from the other by a finite number of applications of elementary Tietze transformations.
\end{prop}

\begin{remark}
The proof of Proposition 3.15 is essentially the same as the proof for its analogue in group theory or semigroup theory.
To see an idea of the proof, one may consult \cite[Theorem 2.5]{Ruskuc1}.
\end{remark}

\begin{corollary}
Let $M$ be a monoid, and let $A$ be an $M$-act defined by a presentation $\langle X\,|\,S\rangle$ where $X$ is finite and $S$ is infinite.  
Then $A$ is finitely presented if and only if there exists a finite subset $S^{\prime}\subseteq S$ such that every relation from $S$ is a consequence of $S^{\prime}$.
\end{corollary}

\begin{proof}
Suppose that $A$ is finitely presented.
By Proposition 3.9, there exists a finite subset $S^{\prime}\subseteq S$ such that $A$ is defined by a finite presentation $\langle X\,|\,S^{\prime}\rangle$.
Therefore, since every relation from $S$ holds in $A$, it must be a consequence of $S^{\prime}$.\par
Conversely, suppose that there exists a finite subset $S^{\prime}\subseteq S$ such that every relation from $S$ is a consequence of $S^{\prime}$.
Using Tietze transformations, we can delete every relation from $S\setminus S^{\prime}$.
By Proposition 3.15, we have that $A$ is defined by the finite presentation $\langle X\,|\,S^{\prime}\rangle$.
\end{proof}

\section{\large{Rees quotients}\nopunct}

\noindent Let $M$ be a monoid, let $A$ be an $M$-act and let $B$ be a subact of $A$.
Recall that the Rees quotient $A/B$ is the quotient act resulting from the Rees congruence $\rho_B$ on $A$ given by
$$a\rho_B b \iff a=b \text{ or } a, b\in B$$
for all $a, b\in A$.  We shall identify the $\rho_B$-class $\{a\}\in A/B$ with $a$ for each $a\in A\setminus B$, and denote the $\rho$-class $B\in A/B$ by $0$.\par 
The purpose of this section is, on the one hand, to construct a presentation for $A/B$ using a presentation for $A$ and a generating set for $B$,
and on the other hand, to derive a presentation for $A$ using presentations for $B$ and $A/B$.
These general presentations will give rise to corollaries pertaining to finite presentability.\par
Let $X$ be any generating set for $A$.
We now give a presentation for $A/B$ in terms of the generators $(X\setminus B)\cup\{0\}$.

\begin{thm}
Let $M$ be a monoid.
Let $A$ be an $M$-act defined by a presentation $\langle X\,|\,R\rangle$, let $B$ be a subact of $A$ generated by $Y$, and let $\langle 0\,|\,S\rangle$ be a presentation for the trivial $M$-act $\{0\}$.
For each $y\in Y$, choose $w_y\in F_X$ such that $y=w_y$ holds in $A$, and let $R^{\prime}=R\cup\{y=w_y : y\in Y\}$.
We now define the sets
\begin{align*}
R_1&=\{(u, v)\in R : u\emph{ represents an element of }A\setminus B\},\\
R_2&=\{u=0 : u\in F_{X\setminus B}\cap L(X, B), (u, v)\in\overline{R^{\prime}}\emph{ for some }v\in F_{X\cup Y}\},
\end{align*}
where $L(X, B)$ denotes the set of elements of $F_X$ which represent elements of $B$.
Then $A/B$ is defined by the presentation $\langle X\setminus B, 0\,|\,R_1, R_2, S\rangle$.
\end{thm}

\begin{proof}
It is clear that $A/B$ satisfies $R_1$, $R_2$ and $S$.
Let $X^{\prime}=X\setminus B$, and let $w_1, w_2\in F_{X^{\prime}\cup\{0\}}$ such that $w_1=w_2$ holds in $A/B$.
By Proposition 3.4, we just need to show that $w_1=w_2$ is a consequence of $R_1$, $R_2$ and $S$.\par
If $w_1$ represents an element of $A\setminus B$, then $w_1=w_2$ is a consequence of $R_1$.\par
Suppose $w_1$ represents $0$ in $A/B$.
We claim that $w_1=0$ is a consequence of $R_1\cup R_2\cup S$.
If $w_1\in F_0$, then $w_1=0$ is a consequence of $S$.
Now assume that $w_1\in F_{X^{\prime}}$.
By Proposition 3.15, we have that $A$ is defined by the presentation $\langle X, Y\,|\, R^{\prime}\rangle$.
Choose $w_1^{\prime}\in F_Y$ such that $w_1=w_1^{\prime}$ holds in $A$.
Then $w_1=w_1^{\prime}$ is a consequence of $R^{\prime}$, so there exists an $R^{\prime}$-sequence
$$w_1=p_1m_1, q_1m_1=p_2m_2, \dots, q_km_k=w_1^{\prime},$$
where $(p_i, q_i)\in\overline{R^{\prime}}$ and $m_i\in M$ for $1\leq i\leq k$.
Note that for $i\in\{1, \dots, k-1\}$, if $p_i$ (and hence $q_i$) represents an element of $A\setminus B$, then $p_{i+1}\in F_{X^{\prime}}$.
Therefore, since $p_1\in F_{X^{\prime}}$ and $p_k$ represents an element of $B$ ($p_k=q_k$ in $A$ and $q_k\in F_Y$),
we may choose $i$ minimal such that $p_i\in F_{X^{\prime}}$ and $p_i$ represents an element of $B.$
We then have that $w_1=p_im_i$ is a consequence of $R_1$, and we obtain $0\cdot m_i$ from $p_im_i$ by an application of a relation from $R_2$.
Now, since $0\cdot m_i=0$ is a consequence of $S$, we deduce that $w_1=0$ is a consequence of $R_1$, $R_2$ and $S.$
This proves the claim.  Exactly the same argument proves that $w_2=0$ is a consequence of $R_1\cup R_2\cup S,$
and hence so is $w_1=0=w_2$, as required.
\end{proof}

\begin{corollary}
Let $M$ be a monoid, let $A$ be a finitely presented $M$-act, and let $B$ be a finitely generated subact of $A$.
Then $A/B$ is finitely presented if and only if the trivial $M$-act is finitely presented (which includes all finitely generated monoids).
\end{corollary}

\begin{proof}
If $A/B$ is finitely presented, then it follows from Lemmma 3.14 that the trivial $M$-act is finitely presented, since $A/B$ contains a zero.
The converse follows immediately from Theorem 4.1.
\end{proof}

We now turn to our second aim in this section: assembling a presentation for $A$ from those for a subact and the Rees quotient.
So, let $M$ be a monoid, let $A$ be an $M$-act and let $B$ be a subact of $A$.
Let $X$ be a generating set for $B$ and let $Y$ be a generating set for $A/B$, and let $Y^{\prime}=Y\setminus\{0\}$.
Note that if $A\setminus B$ is a subact of $A$, the element $0$ must belong to $Y$, and so $Y=Y^{\prime}\cup\{0\}$.
If $A\setminus B$ is not a subact of $A/B$, then there exist elements $w\in F_{Y^{\prime}}$ that represent $0$ in $A/B$, and the set $Y$ need not contain $0$.
We shall now give a presentation for $A$ in terms of the generators $X\cup Y^{\prime}$.

\begin{thm}
Let $M$ be a monoid, let $A$ be an $M$-act and let $B$ be a subact of $A$.
Let $\langle X\,|\,R\rangle$ and $\langle Y\,|\,S\rangle$ be presentations for $B$ and $A/B$ respectively, and let $Y^{\prime}=Y\setminus\{0\}$.
If $A\setminus B$ is not a subact of $A/B$, for each $w\in F_{Y^{\prime}}$ that represents $0$ in $A/B$ choose $\alpha_w\in F_X$ such that $w=\alpha_w$ in $A$, and also fix one of them and denote it by $z$.
We now define the sets
\begin{align*}
S_1&=\{(u, v)\in S : u\emph{ represents an element of }A\setminus B\};\\
S_2&=\{u=\alpha_u: (u, v)\in\overline{S}\emph{ for some }v\in F_Y, u\in F_{Y^{\prime}}, u=0\emph{ in }A/B\}\cup \{z=\alpha_z\}.
\end{align*}
Then $A$ is defined by the presentation $\langle X, Y^{\prime}\,|\,R, S_1, S_2\rangle$.
\end{thm}

\begin{proof}
We first claim that if an element $w\in F_{Y^{\prime}}$ represents an element of $B$ in $A$,
then there exists $w^{\prime}\in F_X$ such that $w=w^{\prime}$ is a consequence of relations from $S_1$ and $S_2$.
Indeed, we have that $w=z$ holds in $A/B$, so $w=z$ is a consequence of $S$; that is, there exists an $S$-sequence 
$$w=p_1m_1, q_1m_1=p_2m_2, \dots, q_km_k=z,$$
where $(p_i, q_i)\in\overline{S}$ and $m_i\in M$ for $1\leq i\leq k$.
If all $(p_i, q_i)\in\overline{S_1}$, then $w=\alpha_z$ is a consequence of $S_1$ and $z=\alpha_z$.
Otherwise, we take $(p_i, q_i)\in\overline{S\setminus S_1}$ with $i$ minimal, so $w=p_im_i$ is a consequence of $S_1$, and we obtain $\alpha_{p_i}m_i$ from $p_im_i$ by an application of a relation from $S_2$.\par
We shall now show that $A$ is defined by the presentation $\langle X, Y^{\prime}\,|\,R, S_1, S_2\rangle$.\par
It is clear that $A$ satisfies $R$, $S_1$ and $S_2$.
Let $w_1, w_2 \in F_{X\cup {Y^{\prime}}}$ be such that $w_1=w_2$ in $A$.
If $w_1$ represents an element of $A\setminus B$, then $w_1=w_2$ is a consequence of $S_1$.
Now suppose that $w_1$ represents an element of $B$.
Using the above claim, if necessary, we have $w_1^{\prime}, w_2^{\prime}\in F_X$ such that $w_1=w_1^{\prime}$ and $w_2=w_2^{\prime}$ are consequences of $S_1$ and $S_2$ (if $w_i\in F_X$, simply let $w_i\equiv w_i^{\prime}$).
But then $w_1^{\prime}=w_2^{\prime}$ holds in $B$, so it is a consequence of $R$.  
Hence, we have that $w_1=w_2$ is a consequence of $R, S_1$ and $S_2$.
\end{proof}

\begin{corollary}
Let $M$ be a monoid, let $A$ be an $M$-act and let $B$ be a subact of $A$.  
If $B$ and $A/B$ are finitely presented, then $A$ is finitely presented.
\end{corollary}

\section{\large{Unions}\nopunct}

\noindent In this section we consider presentations for unions of acts.
A union of acts can be of one of two types: disjoint or amalgamated.
An {\em amalgamated union} of $M$-acts is a union of a family of $M$-acts intersecting pairwise in a common subact.
We only consider the union of two acts, although the results of this section can easily be generalised to any finite number of acts.
Throughout the section we aim to prove our results in the general setting where $C=A\cup B$ is an $M$-act with $A$ and $B$ subacts, and $A\cap B$ is potentially non-empty.
Each of those results will typically have an immediate corollary for disjoint unions, which we state separately immediately after.\par
The main purpose of the section is to explore under what conditions we have $C=A\cup B$ is finitely generated (resp. finitely presented) if and only if $A$ and $B$ are finitely generated (resp. finitely presented), 
and to provide interesting examples to demonstrate that this does not occur in general.
We begin by considering finite generation.

\begin{lemma}
Let $M$ be a monoid, and let $C=A\cup B$ be an $M$-act with $A$ and $B$ subacts of $C$.  
If $A$ and $B$ are finitely generated, then $C$ is finitely generated.
\end{lemma}

\begin{proof}
If $A=\langle X\rangle$ and $B=\langle Y\rangle$, then $C=\langle X\cup Y\rangle$.
\end{proof}

In the following example, we show that the converse to Lemma 5.1 does not hold in general by constructing a finitely generated monoid $M$ and right ideals $A$ and $B$ of $M$ such that $C=A\cup B$ is finitely generated (in fact, finitely presented) but neither $A$ nor $B$ are finitely generated.

\begin{ex}
Let $M=\{a, b\}^{\ast}$.
Let $X=\{a^ib : i\geq 0\}$ and $Y=\{b^ia : i\geq 0\}$, and let $A$ and $B$ be the right ideals generated by $X$ and $Y$ respectively.
It is clear that $A$ and $B$ are not finitely generated.
We have that $C=A\cup B$ is generated by the set $\{a, b\}$ and is free with respect to this generating set, so $C$ is finitely presented.
\end{ex}

\begin{lemma}
Let $M$ be a monoid, let $C=A\cup B$ be an $M$-act with $A$ and $B$ subacts of $C$, and suppose that $A\cap B$ is either empty or finitely generated.  
If $C$ is finitely generated, then both $A$ and $B$ are finitely generated.
\end{lemma}

\begin{proof}
If $A\cap B=\emptyset$, let $U=\emptyset$; otherwise, let $A\cap B=\langle U\rangle$ where $U$ is finite.
Suppose that $C=\langle Z\rangle$.
Let $Y=Z\setminus B$ and let $X=Y\cup U$.  
Let $a\in A$.  If $a\in A\setminus B$, then $a=ym$ for some $y\in Y$ and $m\in M$.
If $a\in A\cap B$, then $a=um$ for some $u\in U$ and $m\in M$.  Therefore, we have that $A=\langle X\rangle$.
Hence, if $Z$ is finite, $A$ is finitely generated, and by symmetry so is $B$.
\end{proof}

\begin{corollary}
Let $M$ be a monoid, and let $A$ and $B$ be disjoint $M$-acts.  Then $A\cup B$ is finitely generated if and only if both $A$ and $B$ are finitely generated.
\end{corollary}

We now turn our attention to finite presentability.  
We begin by giving a general presentation for $C=A\cup B$, 
and we then immediately derive a corollary that gives a sufficient condition for $C$ to be finitely presented.

\begin{thm}
Let $M$ be a monoid, and let $C=A\cup B$ be an $M$-act with $A$ and $B$ subacts of $C$.
Let $A$ and $B$ have presentations $\langle X\,|\,R\rangle$ and $\langle Y\,|\,S\rangle$ respectively.
If $A\cap B\neq\emptyset$, let $U$ be a generating set for $A\cap B$; otherwise, let $U=\emptyset$.
For each $u\in U$, choose $\rho_X(u)\in F_X$ and $\rho_Y(u)\in F_Y$ which both represent $u$ in $C$, and define a set
$$T=\{\rho_X(u)=\rho_Y(u): u\in U\}.$$
Then $C$ is defined by the presentation $\langle X, Y\,|\,R, S, T\rangle$.
\end{thm}

\begin{proof}
Let $w_1, w_2 \in F_{X\cup Y}$ such that $w_1=w_2$ in $C$.\par
If $w_1, w_2 \in F_X$, then $w_1=w_2$ is a consequence of $R$.
If $w_1, w_2 \in F_Y$, then $w_1=w_2$ is a consequence of $S$.\par
Suppose now that $w_1\in F_X$ and $w_2\in F_Y$.  
Let $c=um$, with $u\in U$ and $m\in M$, be the element of $A\cap B$ that both $w_1$ and $w_2$ represent.
Since $w_1=\rho_X(u)m$ holds in $A$, it is a consequence of $R$, and likewise $w_2=\rho_Y(u)m$ is a consequence of $S$.
We also obtain $\rho_Y(u)m$ from $\rho_X(u)m$ by an application of a relation from $T$.
Therefore, we have that $w_1=w_2$ is a consequence of $R, S$ and $T$.
\end{proof}

\begin{corollary}
Let $M$ be a monoid, let $C=A\cup B$ be an $M$-act with $A$ and $B$ subacts of $C$, and suppose that $A\cap B$ is either empty or finitely generated.  
If $A$ and $B$ are finitely presented, then $C$ is finitely presented.
\end{corollary}

The converse to Corollary 5.6 does not hold in general.
Recall that in Example 5.2 we showed that there exists a monoid $M$ with $M$-acts $A$ and $B$ such that $C=A\cup B$ is finitely presented but neither $A$ nor $B$ are finitely generated.
We now present a more striking example:

\begin{ex}
\textit{There exists a monoid $M$ with finitely generated right ideals $A$ and $B$ such that $A\cap B$ is finitely generated and $C=A\cup B$ is finitely presented but neither $A$ nor $B$ are finitely presented}.\par 
Let $M$ be the monoid defined by the presentation
$$\langle a, b, s, t\,|\,ab^ia=aba, ba^ib=bab, sa=a, tb=b~(i\geq 2)\rangle_{\text{Mon}}.$$
We have a complete rewriting system on $X=\{a, b, s, t\}$ consisting of the rules
$$ab^ia\to aba, ~ba^ib\to bab, ~sa\to a, ~tb\to b~(i\geq 2),$$
and this yields the following set of normal forms for $M$:
$$X^{\ast}\setminus\bigl(X^{\ast}(\{ab^ia, ba^ib : i\geq 2\}\cup\{sa, tb\})X^{\ast}\bigr);$$
that is, the set of all words in $X$ which do not contain as a subword the left-hand side of one of the rewriting rules.
For more information on rewriting systems, one may consult \cite{Book} for instance.\par
Let $A$ and $B$ be the right ideals of $M$ generated by $\{a, t\}$ and $\{b, s\}$ respectively.
From the monoid presentation for $M$, we see that $A$ is defined by the infinite presentation 
$$\langle a, t\,|\,a\cdot b^ia=a\cdot ba~(i\geq 2)\rangle.$$
If $A$ were finitely presented, then it could be defined by a presentation
$$P=\langle a, t\,|\,a\cdot b^ia=a\cdot ba~(2\leq i\leq k)\rangle.$$
But if $i>k$, then the relation $a\cdot b^ia=a\cdot ba$ cannot be a consequence of the relations of $P$,
since there do not exist $m, n\in M$ such that $b^ia=mn$ in $M$ and $(a\cdot m, a\cdot ba)\in P$; in other words, no relation of $P$ can be applied to $a\cdot b^ia$.
Therefore, $A$ is not finitely presented.  Similarly, we have that $B$ is not finitely presented.  
We also that $A\cap B=\langle a, b\rangle$.
It is clear from the monoid presentation for $M$ that $C=A\cup B$ is generated by the set $\{s, t\}$ and is free with respect to this generating set, so hence $C$ is finitely presented.
\end{ex}

We now turn to consider conditions for when $C=A\cup B$ being finitely presented implies that the components $A$ and $B$ are both finitely presented.

\begin{thm}
Let $M$ be a monoid, let $C=A\cup B$ be an $M$-act with $A$ and $B$ subacts of $C$, and suppose that $A\cap B$ is either empty or finitely presented.  
If $C$ is finitely presented, then both $A$ and $B$ are finitely presented.
\end{thm}

\begin{proof}
It clearly suffices to show that $A$ is finitely presented.
If $A\cap B=\emptyset$, let $U=\emptyset$; otherwise, let $A\cap B$ be defined by a finite presentation $\langle U\,|\,S\rangle$.
Suppose $C$ is defined by a finite presentation $\langle Z\,|\,R\rangle$.  
Let $Y=Z\setminus B$ and $X=Y\cup U$. As in the proof of Lemma 5.3, we have that $A=\langle X\rangle$.  
Also, let $Y^{\prime}=Z\cap B$.\par
For each $w\in F_Z$ such that $w$ represents an element of $A\cap B$, choose $\rho_U(w)\in F_U$ which represents the same element of $A\cap B$.
Also, for each $u\in U$, choose $w_u\in F_Z$ such that $w_u$ represents $u$.
We now define the following sets:
\begin{align*}
R_1&=\{(u, v)\in R : u, v\in F_Y\};\\
R_2&=\{u=w_u : u\in U, w_u\in F_Y\};\\
R_3&=\{u=\rho_U(u) : u\in F_Y, (u, v)\in\overline{R}\emph{ for some }v\in F_{Y^{\prime}}\}.
\end{align*}
Note that the set $R_2$ may be empty (if $w_u\in F_{Y^{\prime}}$ for every $u\in U$); however, this does not affect the argument that follows.
We make the following claim:
\begin{claim*}
If an element $w\in F_Y$ represents an element of $A\cap B$, 
then there exists $w^{\prime}\in F_U$ such that $w=w^{\prime}$ is a consequence of relations from $R_1, R_2$ and $R_3$.
\end{claim*}
\begin{proof}
Let $w\in F_Y$ represent an element $c\in A\cap B$.
Now $c=um$ for some $u\in U$ and $m\in M$.
Since $w=w_um$ holds in $C$, it is a consequence of $R$, so there exists an $R$-sequence
$$w=p_1m_1, q_1m_1=p_2m_2, \dots, q_km_k=w_um,$$
where $(p_i, q_i)\in\overline{R}$ and $m_i\in M$ for $1\leq i\leq k$.
If each $(p_i, q_i)\in\overline{R_1}$, then $w=w_um$ is a consequence of $R_1$,
and we obtain $u\cdot m$ from $w_um$ by an application of a relation from $R_2$.
Otherwise, there exists $i$ minimal such that $p_i\in F_Y$ and $q_i\in F_{Y^{\prime}}$, 
so $w=p_im_i$ is a consequence of $R_1$, and we obtain $\rho_U(p_i)m_i$ from $p_im_i$ by an application of a relation from $R_3$.
\end{proof}
Returning to the proof of Theorem 5.8, we shall show that $A$ is defined by the finite presentation $\langle X\,|\,R_1, R_2, R_3, S\rangle$.\par
Let $w_1, w_2\in F_X$ such that $w_1=w_2$ holds in $A$.
If $w_1$ represents an element of $A\setminus B$, then $w_1=w_2$ is a consequence of $R_1$.
Suppose $w_1$ represents an element of $A\cap B$.
Using the claim above, if necessary, we have $w_1^{\prime}$, $w_2^{\prime}\in F_U$ such that $w_1=w_1^{\prime}$ and $w_2=w_2^{\prime}$ are consequences of $R_1, R_2$ and $R_3$ (if $w_i\in F_U$, simply let $w_i\equiv w_i^{\prime}$).
Since $w_1^{\prime}=w_2^{\prime}$ holds in $A\cap B$, we have that $w_1^{\prime}=w_2^{\prime}$ is a consequence of $S$.
Therefore, $w_1=w_2$ is a consequence of $R_1, R_2, R_3$ and $S$.
\end{proof}

\begin{corollary}
Let $M$ be a monoid, and let $A$ and $B$ be disjoint $M$-acts.  Then $A\cup B$ is finitely presented if and only if both $A$ and $B$ are finitely presented.
\end{corollary}

We now investigate how finite presentability of $C=A\cup B$ affects the intersection $A\cap B$.

\begin{prop}
Let $M$ be a monoid, let $C=A\cup B$ be an $M$-act with $A$ and $B$ subacts of $C$, and suppose that $A\cap B$ is non-empty.  
If $C$ is finitely presented and both $A$ and $B$ are finitely generated, then $A\cap B$ is finitely generated.
\end{prop}

\begin{proof}
Let $A$ and $B$ be generated by finite sets $X$ and $Y$ respectively.
Since $C$ is finitely presented, it can be defined by a finite presentation $\langle X, Y\,|\,R\rangle$.\par
For any $w\in F_X$, let $\overline{w}$ denote the element of $A$ which $w$ represents, and define
$$U=\{\overline{u} : u\in F_X, (u, v)\in\overline{R}\text{ for some } v\in F_Y\}\subseteq A\cap B.$$
Let $c\in A\cap B$.  Choose $w_1\in F_X$ and $w_2\in F_Y$ which both represent the element $c$.
Since $w_1=w_2$ holds in $C$, it is a consequence of $R$, so there exists an $R$-sequence
$$w_1=p_1m_1, q_1m_1=p_2m_2, \dots, q_km_k=w_2,$$
where $(p_i, q_i)\in\overline{R}$ and $m_i\in M$ for $1\leq i\leq k$.
Now, there exists $i\in\{1, \dots, k\}$ such that $p_i\in F_X$ and $q_i\in F_Y$,
and we have that $c=\overline{p_i}m_i$, so $c\in\langle U\rangle$.  
Hence, we have that $A\cap B=\langle U\rangle$, so $A\cap B$ is finitely generated.
\end{proof}

\begin{corollary}
Let $M$ be a monoid, and let $C=A\cup B$ be an $M$-act with $A$ and $B$ finitely presented subacts of $C$ and $A\cap B$ non-empty.
Then $C$ is finitely presented if and only if $A\cap B$ is finitely generated.
\end{corollary}

\begin{remark}
In the catagorical sense, the disjoint union of acts is a {\em coproduct} and the amalgamated union of acts is a {\em pushout}.
In the category of groups, the coproduct is called the {\em free product} and the pushout is called the {\em free product with amalgamation}.
Notice the similarity between Corollary 5.11 and a well-known result, due to G. Baumslag, 
which states that for two finitely presented groups $G_1$ and $G_2$ such that $H=G_1\cap G_2$ is a group, 
the amalgamated free product $G=G_1\ast_H G_2$ is finitely presented if and only if $H$ is finitely generated;
see \cite[Chapter 6]{Baumslag} for more details.
\end{remark}

Given the results concerning finite presentability in this section, the following two questions arise:
Do there exist monoids $M$ and $M$-acts $A$, $B$ and $C$ with $C=A\cup B$, such that:
\begin{enumerate}
\item $A$ and $B$ are finitely presented but $C$ is not finitely presented?
\item $A$ and $B$ are finitely presented, while $A\cap B$ is finitely generated but not finitely presented?
\end{enumerate}
In the following, we exhibit examples which provide positive answers to both of the above two questions.

\begin{ex}
\textit{There exists a monoid $M$ with finitely presented right ideals $A$ and $B$ such that $C=A\cup B$ is not finitely presented}.\par 
Let $M$ be the monoid defined by the presentation
$$\langle a, b, c\,|\,ac^ia=bc^{i-1}b~(i\geq 2)\rangle_{\text{Mon}}.$$
We have a complete rewriting system on $X=\{a, b, c\}$ consisting of the rules
$$ac^ia\to bc^{i-1}b~(i\geq 2),$$
and this yields the following set of normal forms for $M$:
$$X^{\ast}\setminus(X^{\ast}\{ac^ia : i\geq 2\}X^{\ast}).$$
Let $A$ and $B$ be the right ideals of $M$ generated by $\{a\}$ and $\{b\}$ respectively.
We have that $A$ and $B$ are free $M$-acts and hence finitely presented. 
Let $X=\{ac^ia : i\geq 2\}$.
It is clear from the monoid presentation for $M$ that $X$ generates $A\cap B$,
and that this is a minimal generating set for $A\cap B$, so $A\cap B$ is not finitely generated.
It now follows from Corollary 5.11 that $C=A\cup B$ is not finitely presented.
\end{ex}

\begin{ex}
\textit{There exists a monoid $M$ with finitely presented right ideals $A$ and $B$ such that $A\cap B$ is finitely generated but not finitely presented}.\par 
Let $M$ be the monoid defined by the presentation
$$\langle a, b, c\,|\,a^2=a, cab=ab, ab^ia=aba~(i\geq 2)\rangle_{\text{Mon}}.$$
We have a complete rewriting system on $X=\{a, b, c\}$ consisting of the rules
$$a^2\to a, ~cab\to ab, ~ab^ia\to aba~(i\geq 2),$$
and this yields the following set of normal forms for $M$:
$$X^{\ast}\setminus\bigl(X^{\ast}(\{ab^ia : i\geq 2\}\cup\{a^2, cab\})X^{\ast}\bigr).$$
Let $A$ and $B$ be the right ideals of $M$ generated by $\{a\}$ and $\{c\}$ respectively.
From the monoid presentation for $M$, we see that $B$ is a free $M$-act (and hence finitely presented), 
that $A$ is defined by the infinite presentation 
$$\langle a\,|\,a\cdot a=a, a\cdot b^ia=a\cdot ba~(i\geq 2)\rangle,$$
and that $A\cap B$ is defined by the infinite presentation 
$\langle y\,|\,y\cdot b^ia=y\cdot a~(i\in\mathbb{N})\rangle$, where $y$ respesents $ab$.  
We claim that $A$ is also defined by the finite presentation $\langle a\,|\, a\cdot a=a\rangle$.
Indeed, for any $i\geq 2$, we have
$$a\cdot b^ia=(a\cdot a)b^ia\equiv a\cdot ab^ia\equiv a\cdot aba\equiv(a\cdot a)ba\textcolor{red}{=}a\cdot ba.$$
It can be shown that $A\cap B$ is not finitely presented using a similar argument to the one in Example 5.7.\par
Note that since $A$ and $B$ are finitely presented and $A\cap B$ is finitely generated, Corollary 5.11 implies that $C=A\cup B$ is finitely presented.
In fact, it is easy to see that $C$ is defined by the finite presentation 
$$\langle a, c\,|\,a\cdot a=a, a\cdot b=c\cdot ab\rangle.$$
\end{ex}

\section{\large{Subacts}\nopunct}

\noindent In this section we consider presentations for subacts of monoid acts.
In the first part of the section we construct a general (infinite) presentation for a subact of a monoid act.
From this presentation we obtain a method for finding `nicer' presentations in special situations.
We note that for general monoids $M$, finitely generated subacts of finitely presented $M$-acts are not necessarily finitely presented.\par
In the second part of the section, we shall consider a particular case where we have a subact $B$ with finite complement in an $M$-act $A$;
we say that $B$ is {\em large} in $A$ and $A$ is a {\em small extension} of $B$.
This was motivated by the analagous concept of `large subsemigroups' within semigroup theory; see \cite{Ruskuc2} for more details.
In particular, it is shown there that various finiteness properties, including finite generation and finite presentability, are inherited by both large subsemigroups and small extensions of semigroups.
Given these results, it is natural to ask whether similar results hold in the setting of monoid acts.
We shall show that, for finitely generated monoids $M$, finite generation is inherited by both large subacts and small extensions, 
and finite presentability is also inherited by small extensions.
Somewhat surprisingly, though, there exist finitely generated monoids $M$ for which large subacts of finitely presented $M$-acts are not necessarily finitely presented.
We shall show, however, that there is a large class of monoids for which finite presentability is inherited by large subacts.\par
~\par
Let $M$ be a monoid, let $A$ be an $M$-act defined by a presentation $\langle X\,|\,R\rangle$, and let $B$ be a subact of $A$ generated by a set $Y$.
We seek a presentation for $B$ in terms of the generators $Y$.\par
For each $y\in Y$, we choose $w_y\in F_X$ which represents $y$,
and let $\psi : F_Y \to F_X$ be the unique $M$-homomorphism extending $y\mapsto w_y$.  
We call $\psi$ the {\em representation map}.
For an element $w\in F_X$ which represents an element of $B$, we have $w=x\cdot m$ for some $x\in X$ and $m\in M$, and $xm=yn$ for some $y\in Y, n \in M$, so $w=(y\psi)n$ holds in $B$.
Therefore, we have a map $\phi : L(X, B)\to F_Y$, where $L(X, B)$ denotes the set of all elements of $F_X$ which represent elements of $B$, satisfying $(w\phi)\psi=w$ in $A$ for all $w\in L(X, B)$.
We call $\phi$ a {\em rewriting map}.  Note that the existence of $\phi$ follows from the Axiom of Choice.\par
We now state our first result of this section, giving a presentation for $B$, which has analogues within group and semigroup theory; 
see \cite[Theorem 2.6]{Magnus} and \cite[Theorem 2.1]{Campbell} for more details.

\begin{thm}
Let $M$ be a monoid, let $A$ be an $M$-act defined by a presentation $\langle X\,|\,R\rangle,$ and let $B$ be a subact of $A$ generated by $Y$.
For each $y\in Y$, we choose $w_y\in F_X$ which represents $y$.
Let $\psi$ be the representation map and let $\phi$ be a rewriting map, and define the following sets of relations:
\begin{align*}
R_1&=\{y=w_y\phi : y\in Y\};\\
R_2&=\{(wm)\phi=(w\phi)m : w\in L(X, B), m\in M\};\\
R_3&=\{(um)\phi=(vm)\phi : (u, v)\in R, m\in M, um\in L(X, B)\}.
\end{align*}
Then $B$ is defined by the presentation $\langle Y\,|\,R_1, R_2, R_3\rangle$.
\end{thm}

\begin{proof}
We first show that $B$ satisfies $R_1$, $R_2$ and $R_3$.  
This amounts to showing that $u\psi=v\psi$ holds in $A$ for each $u=v$ in $R_1, R_2$ and $R_3$.\par
For each $y\in Y$, we have $y\psi\equiv w_y=(w_y\phi)\psi$ holds in $A$, since $w_y\in L(X, B)$.
For any $w\in L(X, B)$ and $m\in M$, we have
$$((wm)\phi)\psi=wm=((w\phi)\psi)m\equiv ((w\phi)m)\psi$$
holds in $A$.  Finally, for any $(u, v)\in R, m\in M$ such that $um\in L(X, B)$, we have
$$((um)\phi)\psi=um=vm=((vm)\phi)\psi$$
holds in $A$.\par
We now claim that for any $w\in F_Y$, we have that $w=(w\psi)\phi$ is a consequence of $R_1$ and $R_2$.
Indeed, we have $w\equiv y\cdot m$ for some $y\in Y$ and $m \in M$, so $w\psi\equiv w_ym$.
We obtain $(w_y\phi)m$ from $w$ by an application of the relation $y=w_y\phi$,
and since $w_y\in L(X, B)$, we have that $(w\psi)\phi=(w_y\phi)m$ is a relation from $R_2$.\par
Now let $w_1, w_2\in F_Y$ be such that $w_1=w_2$ holds in $B$. 
Since $w_1\psi=w_2\psi$ holds in $A$, it is a consequence of $R$, so we have an $R$-sequence 
$$w_1\psi=p_1m_1, q_1m_1=p_2m_2, \dots, q_km_k=w_2,$$
where $(p_i, q_i)\in R$ and $m_i\in M$ for $1\leq i\leq k$.  
For each $i\in\{1, \dots, k\}$, we have $p_im_i\in L(X, B)$, so $(w_1\psi)\phi=(w_2\psi)\phi$ is a consequence of the relations $(p_im_i)\phi=(q_im_i)\phi$ of $R_3$.
Finally, since $w_1=(w_1\psi)\phi$ and $(w_2\psi)\phi=w_2$ are consequences of $R_1$ and $R_2$, 
we conclude that $w_1=w_2$ is a consequence of $R_1$, $R_2$ and $R_3$.
\end{proof}

\begin{remark}
The presentation from Theorem 6.1 has the disadvantage that it always has infinitely many relations if $M$ is an infinite monoid,
and neither the rewriting map nor the set $L(X, B)$ have been defined constructively.
However, the result does give a method for finding `nice' presentations for subacts in certain cases.
Given an $M$-act $A$ defined by a presentation $\langle X\,|\,R\rangle$ and a subact $B$ of $A$, this method consists of the following:
\begin{enumerate}
 \item finding a generating set $Y$ for $B$;
 \item finding a rewriting map $\phi : L(X, B)\to F_Y$;
 \item finding a set $S\subseteq F_Y\times F_Y$ of relations which hold in $B$ and imply the relations of the presentation, say $P$, given in Theorem 6.1.
\end{enumerate}
Using Tietze transformations, we can add $S$ to $P$ ($S$ must be a consequence of the relations of $P$ since these are defining relations for $B$) and then remove the remaining relations (since they are consequences of $S$).
Hence, by Proposition 3.15, we have that $B$ is defined by the presentation $\langle Y\,|\,S\rangle$.
\end{remark}
~\par
For the remainder of this section we shall be considering large subacts. 
Recall that a subact $B$ of an $M$-act is said to be {\em large} in $A$, and $A$ is said to be a {\em small extension} of $B$, if the set $A\setminus B$ is finite.
We shall investigate how similar an $M$-act $A$ and a large subact $B$ of $A$ are with regard to finite generation and finitely presentability.
We begin by considering finite generation.\par

\begin{lemma}
Let $M$ be a monoid, let $A$ be an $M$-act, and let $B$ be a large subact of $A$.  
If $B$ is finitely generated, then $A$ is finitely generated.
\end{lemma}

\begin{proof}
If $B$ is generated by a set $X$, then $A$ is generated by $X\cup(A\setminus B)$.
\end{proof}

In the following, we show that the converse to Lemma 6.3 does not hold for monoids in general, but it does however hold for all groups and all finitely generated monoids.

\begin{ex}
Let $M=X^{\ast}$ with $X$ infinite.  Let $I=X^{+}$, so $I$ is a large subact of the cyclic $M$-act $M$.
Clearly $X$ is a minimal generating set for $I$, so $I$ is not finitely generated.
\end{ex}

\begin{lemma}
Let $M$ be a group, let $A$ be an $M$-act, and let $B$ be a large subact of $A$.  
If $A$ is finitely generated, then $B$ is finitely generated.
\end{lemma}

\begin{proof}
Since $A$ is the disjoint union of its subacts $B$ and $A\setminus B$, 
it follows from Corollary 5.4 that $B$ is finitely generated if $A$ is finitely generated.
\end{proof}

The following result provides a generating set for a subact.

\begin{prop}
Let $M$ be a monoid generated by a set $Z$, let $A$ be an $M$-act generated by a set $X$, and let $B$ be a subact of $A$.  
Define a set $$S=\{am\in B : a\in A\setminus B, m\in Z\},$$ and let $Y=(X\cap B)\cup S$.
Then $B$ is generated by the set $Y.$
\end{prop}

\begin{proof}
Since $Y\subseteq B$ and $B$ is a subact of $A$, we have $\langle Y\rangle\subseteq B$.
Let $b\in B$.  If $b=xm$ for some $x\in X\cap B, m\in M$, then $b\in\langle Y\rangle$.
Otherwise, $b=xm_1\dots m_k$ for some $x\in X\setminus B, m_i\in Z$.
Let $s$ be minimal such that $xm_1\dots m_s\in B$, and let $a=xm_1\dots m_{s-1}$.  Then $a\in A\setminus B$ and $am_s\in B$, so $am_s\in S$.
Therefore, we have $$b=(am_s)m_{s+1}\dots m_k\in\langle S\rangle\subseteq\langle Y\rangle.$$  
Hence, we have that $B=\langle Y\rangle$.
\end{proof}

\begin{corollary}
Let $M$ be a finitely generated monoid, let $A$ be an $M$-act, and let $B$ be a large subact of $A.$  
If $A$ is finitely generated, then $B$ is finitely generated.
\end{corollary}

We have shown that, for groups and finitely generated monoids $M,$ finite generation is inherited by both large subacts and small extensions.
We now turn our attention to finite presentability.
For a monoid $M$, there are two questions relating to large subacts that arise:
\begin{enumerate}
\item Is every small extension of every finitely presented $M$-act finitely presented?
\item Is every large subact of every finitely presented $M$-act finitely presented?
\end{enumerate}

We first show that the property \textit{every small extension of every finitely presented $M$-act is finitely presented} is equivalent to another monoid property. 
From this we immediately derive as a corollary that finitely generated monoids $M$ admit a positive answer to question (1).

\begin{prop}
The following are equivalent for a monoid $M$:
\begin{enumerate}
 \item every finite $M$-act is finitely presented;
 \item every small extension of every finitely presented $M$-act is finitely presented.
\end{enumerate}
\end{prop}

\begin{proof}
$(1)\Rightarrow(2)$.  Let $A$ be a small extension of a finitely presented $M$-act $B$.
We have that $A/B$ is finite and hence finitely presented by assumption.
Since $B$ and $A/B$ are finitely presented, it follows that $A$ is finitely presented by Corollary 4.4.
\\$(2)\Rightarrow(1)$.  Let $A$ be a finite $M$-act.
Choose a finitely presented $M$-act $B$ disjoint from $A$.
We have that $A\cup B$ is a small extension of $B$, so it is finitely presented by assumption.
Hence, by Corollary 5.9, we have that $A$ is finitely presented.
\end{proof}

The above result together with Lemma 3.11 yields:

\begin{corollary}
Let $M$ be a finitely generated monoid.
Then every small extension of every finitely presented $M$-act is finitely presented.
\end{corollary}

We now turn to consider which monoids $M$ give a positive answer to question (2); that is, every large subact of every finitely presented $M$-act is finitely presented.
We first present an example which reveals that there exist finitely generated monoids which do not possess this property,
and then we show that there exists a large class of monoids for which the property holds.

\begin{ex}
Let $M$ be the monoid defined by the presentation
$$\langle a, b\,|\,ab^ia=aba~(i\geq 2)\rangle_{\text{Mon}}.$$
Let $I=M\setminus\{1, a\}$, so $I$ is a large subact of the finitely presented $M$-act $M.$\par
Now $I=\langle b, a^2, ab\rangle$ is a disjoint union of $\langle b\rangle, \langle a^2\rangle$ and $\langle ab\rangle.$
Clearly $\langle b\rangle$ and $\langle a^2\rangle$ are free and hence finitely presented.
Letting $y=ab$, we have that $\langle y\rangle$ is defined by the presentation 
$$\langle y\,|\,y\cdot b^ia=y\cdot a~(i \in \mathbb{N})\rangle.$$
We saw in Example 5.14 that $\langle y\rangle$ is not finitely presented.
It hence follows from Corollary 5.9 that $I$ is not finitely presented.
\end{ex}

\begin{lemma}
Let $M$ be a group, let $A$ be an $M$-act and let $B$ be a large subact of $A$.  
If $A$ is finitely presented, then $B$ is finitely presented.
\end{lemma}

\begin{proof}
Since $A$ is the disjoint union of its subacts $B$ and $A\setminus B$, 
it follows from Corollary 5.9 that $B$ is finitely presented if $A$ is finitely presented.
\end{proof}

\begin{defn}
A monoid $M$ is {\em right coherent} if every finitely generated subact of every finitely presented $M$-act is finitely presented.
\end{defn}

Examples of right coherent monoids include groups, Clifford monoids, semilattices, the bicyclic monoid, free commutative monoids, free monoids and the free left ample monoid; see \cite{Gould1,Gould2,Gould3,Gould4}.\par
Since for any finitely generated monoid $M,$ a large subact of a finitely generated $M$-act is finitely generated, we have the following result:

\begin{lemma}
Let $M$ be a finitely generated right coherent monoid, let $A$ be an $M$-act and let $B$ be a large subact of $A$.  
If $A$ is finitely presented, then $B$ is finitely presented.
\end{lemma}

Before stating our final result of this section, we first introduce a technical definition.\par
Let $M$ be a monoid with a presentation $\langle Z\,|\,P\rangle_{\text{Mon}}$, and let $A$ be an $M$-act with a presentation $\langle X\,|\,R\rangle.$
For a word $w$ in $Z^{\ast},$ let $\overline{w}$ denote the element of $M$ which $w$ represents.
We say an element $x\cdot w\in F_{X, Z^{\ast}}$ {\em represents} an element $a\in A$ if $x\cdot\overline{w}\in F_{X, M}$ represents $a\in A.$

\begin{thm}
Let $M$ be a finitely presented monoid, let $A$ be an $M$-act, and let $B$ be a large subact of $A$.  
If $A$ is finitely presented, then $B$ is finitely presented.
\end{thm}

\begin{proof}
We shall prove this result using the method based on Theorem 6.1 and outlined in Remark 6.2.\par
Let $M$ be defined by the presentation $\langle Z\,|\,P\rangle_{\text{Mon}}$, where $Z$ and $P$ are finite.
Suppose $A$ is defined by the finite presentation $\langle X\,|\,R\rangle$.
We define the finite set $$S=\{am\in B : a\in A\setminus B, m\in Z\},$$ 
and let $Y=(X\cap B)\cup S$.
We have that $B=\langle Y\rangle$ by Proposition 6.6.\par
Let $W$ denote the set of elements of $F_{X, Z^{\ast}}$ which represent elements of $B$.
We define a map $$\theta : W\to F_{Y, M}$$ as follows.
For $u\in W,$ we have $u=x\cdot w$ for some $x\in X$ and $w\in Z^{\ast}.$
If $x\in B,$ let $u\theta=x\cdot\overline{w}.$ 
Suppose $x\in A\setminus B.$  We have $w=m_1\dots m_k$ with $m_i\in Z$.
Let $s$ be minimal such that $x\cdot m_1\dots m_s$ represents an element of $B$, say $b$, and let $u\theta=b\cdot\overline{m_{s+1}\dots m_k}$.\par
Note that $(uw)\theta\equiv(u\theta)\overline{w}$ for all $u\in W$ and $w\in Z^{\ast}.$\par  
Let $L(X, B)$ denote the set of elements of $F_{X, M}$ which represent elements of $B$.
For each $m\in M$, choose an element $w_m\in Z^{\ast}$ which represents $m$.
We now have a well-defined rewriting map
$$\phi : L(X, B)\to F_Y, x\cdot m\mapsto(x\cdot w_m)\theta.$$
Note that for any $x\in X\cap B$, we have $x\equiv x\phi$.
Now, for each $y\in S,$ we have $y=a_ym_y$ for some $a_y\in A\setminus B$ and $m_y\in Z$.
Choose $u_y\in F_X$ which represents $a_y$ in $A$, so $(u_ym_y)\phi=y$ holds in $B$.\par
We now define the following sets of relations:
\begin{align*}
S_1=\,&\{u\phi=v\phi : (u, v)\in R, u\in L(X, B)\};\\
S_2=\,&\{b\cdot\overline{w}=c\cdot\overline{z} : b, c\in S, w\text{ and }z\text{ are suffixes of }p\text{ and }q\text{ respectively}\\& \text{ for some } (p, q) \in P, b\cdot\overline{w}=c\cdot\overline{z} \text{ holds in } B\}.
\end{align*}
Since $R$, $S$ and $P$ are finite, we have that $S_1$ and $S_2$ are finite.\par 
We now make the following claim:
\begin{claim*}
Let $x\in X$ and $w, w^{\prime}\in Z^{\ast}$ such that $w=w^{\prime}$ holds in $M$ and $x\cdot w$ represents an element of $B.$
Then $(x\cdot w)\theta=(x\cdot w^{\prime})\theta$ is a consequence of $S_2$.
\end{claim*}
\begin{proof}
If $x\in B,$ then 
$$(x\cdot w)\theta\equiv x\cdot\overline{w}\equiv x\cdot\overline{w^{\prime}}\equiv(x\cdot w^{\prime})\theta.$$
Suppose now that $x\in A\setminus B$.  
Since $w=w^{\prime}$ is a consequence of $P,$ it is clearly sufficient to consider the case where $w^{\prime}$ is obtained from $w$ by a single application of a relation from $P,$ 
so let $w=pqr$ and $w^{\prime}=pq^{\prime}r$ where $p, r\in Z^{\ast}$ and $(q, q^{\prime})\in P.$
There are three cases.\par
\noindent{\em Case} 1: $x\cdot p$ represents an element of $B.$  Since $q=q^{\prime}$ in $M,$ we have 
$$(x\cdot w)\theta\equiv((x\cdot p)\theta)\overline{qr}\equiv((x\cdot p)\theta)\overline{q^{\prime}r}\equiv(x\cdot w^{\prime})\theta.$$
{\em Case} 2: $x\cdot pq$ represents an element of $A\setminus B.$  
Now $r=m_1\dots m_k$ where $m_i\in Z.$
Let $s$ be minimal such that $x\cdot pqm_1\dots m_s$ represents an element of $B$, say $b$.
Since $x\cdot pq$ and $x\cdot pq^{\prime}$ represent the same element of $A$, we have
$$(x\cdot w)\theta\equiv b\cdot\overline{m_{s+1}\dots m_k}\equiv(x\cdot w^{\prime})\theta.$$
{\em Case} 3: $x\cdot p$ represents an element of $A\setminus B$ and $x\cdot pq$ represents an element of $B$.
Now $q=m_1\dots m_k$ and $q^{\prime}=n_1\dots n_l$ where $m_i, n_i\in Z$.
Let $s$ be minimal such that $x\cdot pm_1\dots m_s$ represents an element of $B$, say $b$, 
and let $t$ be minimal such that $x\cdot pn_1\dots n_t$ represents an element of $B$, say $c$.
We have that
$$(x\cdot w)\theta\equiv(b\cdot\overline{m_{s+1}\dots m_k})\overline{r}=(c\cdot\overline{n_{t+1}\dots n_l})\overline{r}\equiv(x\cdot w^{\prime})\theta,$$
using an application of a relation from $S_2$.
\end{proof}
Returning to the proof of Theorem 6.14, we shall show that $B$ is defined by the finite presentation $\langle Y\,|\,S_1, S_2\rangle$.
We need to show that the relations $R_1$, $R_2$ and $R_3$ of the presentation for $B$ given in Theorem 6.1 are consequences of $S_1$ and $S_2$. 
That is, we show that for any $y\in S$, $w\in L(X, B)$ and $m\in M$,
and $(u, v)\in R$, $n\in M$ such that $un\in L(X, B)$,
the relations $y=(u_ym_y)\phi$, $(wm)\phi=(w\phi)m$ and $(un)\phi=(vn)\phi$ are consequences of $S_1$ and $S_2$.\par
Let $y\in S$.  We have that $u_y=x\cdot m$ for some $x\in X\setminus B$ and $m\in M$.
Since $w_mm_y=w_{mm_y}$ holds in $M$, we have that 
$$y\equiv(x\cdot w_mm_y)\theta=(x\cdot w_{mm_y})\theta\equiv(u_ym_y)\phi$$ 
is a consequence of $S_2$ by the above claim.\par
Now let $u=x\cdot n\in L(X, B)$ and $m\in M.$
We have that $(um)\phi\equiv(x\cdot w_{nm})\theta$ and $(u\phi)m\equiv(x\cdot w_nw_m)\theta$.
Since $w_{nm}=w_nw_m$ holds in $M,$ we have that $(um)\phi=(u\phi)m$ is a consequence of $S_2$ by the above claim.\par
Finally, let $(u, v)\in R$ and $n\in M$ such that $un\in L(X, B)$.
Suppose first that $u\in L(X, B)$.  We have that $(un)\phi=(u\phi)n$ and $(vn)\phi=(v\phi)n$ are consequences of $S_2$,
and we obtain $(v\phi)n$ from $(u\phi)n$ by an application of a relation from $S_1$.
Therefore, $(un)\phi=(vn)\phi$ is a consequence of $S_1$ and $S_2$.\par 
Suppose now that $u$ represents an element of $A\setminus B$.
We have that $u=x\cdot m$ and $v=x^{\prime}\cdot m^{\prime}$ for some $x, x^{\prime}\in X$ and $m, m^{\prime}\in M.$
Now $(un)\phi=(x\cdot w_mw_n)\theta$ and $(vn)\phi=(x^{\prime}\cdot w_{m^{\prime}}w_n)\theta$ are consequences of $S_2$ by the above claim.
We have that $w_n=m_1\dots m_k$ where $m_i\in Z.$
Let $s$ be minimal such that $x\cdot w_mm_1\dots m_s$ represents an element of $B$, say $b$.
Since $x\cdot w_m$ and $x^{\prime}\cdot w_{m^{\prime}}$ represent the same element of $A$, we have
$$(x\cdot w_mw_n)\theta\equiv b\cdot\overline{m_{s+1}\dots m_k}\equiv(x^{\prime}\cdot w_{m^{\prime}}w_n)\theta.$$
Therefore, we have that $(un)\phi=(vn)\phi$ is a consequence of $S_2$.
\end{proof}

\end{document}